\documentclass[11pt,a4paper]{article}

\usepackage{color}

\usepackage{theorem,enumerate}
\usepackage{amsmath,latexsym,amssymb,amsfonts}
\usepackage{eucal}
\usepackage{mathrsfs}
\usepackage{array}
\usepackage{footmisc}
\usepackage{graphicx, color}
\usepackage{caption,subcaption}

\newtheorem{theorem}{Theorem} \rm
\newtheorem{lemma}[theorem]{Lemma}

\newtheorem{corollary}[theorem]{Corollary}

\newtheorem{definition}[theorem]{Definition}

\newtheorem{question}[theorem]{Question}
\numberwithin{theorem}{section}

\newcommand{\be}{\beta}
\newcommand{\ga}{\gamma}
\newcommand{\om}{\omega}
\newcommand{\al}{\alpha}

\newcommand{\qed}{\hfill \ensuremath{\Box}}

\begin{document}
\title{\bf The Alon-Tarsi number of subgraphs of a planar graph}

\author{
Ringi Kim\thanks{Department of Mathematical Sciences, KAIST, Korea.
e-mail: {\tt kimrg@kaist.ac.kr}}
\thanks{This work was supported by the National Research Foundation of Korea(NRF) grant funded by the Korea government(MSIT)(NRF-2018R1C1B6003786)}
,
Seog-Jin Kim\thanks{Department of Mathematics Education, Konkuk University,
Korea.
e-mail: {\tt skim12@konkuk.ac.kr
}} \thanks{
This work was supported by the National Research Foundation of Korea(NRF) grant funded by the Korea government(MSIT) (NRF-2018R1A2B6003412).},
Xuding Zhu\thanks{Department of Mathematics, Zhejiang Normal University, China
e-mail: {\tt  xdzhu@zjnu.edu.cn
}} \thanks{Grant Numbers: NSFC 11571319 and 111 project of Ministry of Education of China.}
}

\maketitle

\begin{abstract}
This paper constructs a planar graph $G_1$ such that for any subgraph $H$ of $G_1$ with maximum degree $\Delta(H) \le 3$, $G_1-E(H)$ is not $3$-choosable, and a planar graph $G_2$ such that for any star forest $F$ in $G_2$,  $G_2-E(F)$ contains a copy of $K_4$ and hence $G_2-E(F)$ is not $3$-colourable.  
On the other hand, we prove that every planar graph $G$ contains a forest $F$ such that the Alon-Tarsi number of $G - E(F)$ is at most $3$, and hence     $G - E(F)$ is 3-paintable and 3-choosable.   
\end{abstract}

\section{Introduction}

Assume $G$ is a graph and $d$ is  a non-negative integer. A  {\em $d$-defective colouring} of $G$ is a colouring $\phi$ of the vertices of $G$ such that each colour class induces a subgraph of maximum degree at most $d$. A $0$-defective colouring of $G$ is also called a {\em proper colouring} of $G$.

A \emph{$k$-list assignment} of a graph $G$ is a mapping $L$ which assigns to each vertex $v$ of $G$ a set $L(v)$ of $k$ permissible colours.  Given a $k$-list assignment $L$ of $G$, a  \emph{$d$-defective $L$-colouring} of G is a $d$-defective colouring $\phi$ of $G$ such that  $\phi(v) \in L(v)$
for each vertex $v$ of $G$.   We say $G$ is \emph{$d$-defective $k$-choosable} if $G$ has a $d$-defective $L$-colouring for every $k$-list assignment $L$. We say $G$ is {\em $k$-choosable } if $G$ is $0$-defective $k$-choosable.
The {\em   choice number $\chi_{\ell}(G)$} of a graph $G$ is defined as the least integer $k$ such that $G$ is $k$-choosable.

Defective list colouring of graphs has been studied a lot in the literature. It was proved in \cite{CCW} that every outerplanar graph is $2$-defective $2$-colourable and every planar graph is $2$-defective $3$-colourable. These results   were generalized  in \cite{EH} and \cite{Skrekovski1999}, where the authors  proved independently that every outerplanar graph is $2$-defective $2$-choosable and every planar graph is $2$-defective $3$-choosable.  Both papers \cite{EH} and \cite{Skrekovski1999} asked the question whether every planar graph is $1$-defective $4$-choosable. One decade later, Cushing and Kierstead \cite{CKierstead} answered this question in affirmative. 

On-line version of defective list colouring was first studied in \cite{HZ2016}. It is defined through a two-person game. Given a graph $G$ and non-negative integers $d, k$, the {\em $d$-defective $k$-painting game on $G$} is played by two players: Lister and Painter. Initially, each vertex of $G$ has $k$ tokens and is uncoloured. In each round, Lister selects a set $U$ of uncoloured vertices and takes away one token from each vertex in $U$. Painter selects a subset $X$ of $U$ such that the induced subgraph $G[X]$ has maximum degree at most $d$, and colours all the vertices of $X$. If at the end of some round, there is an uncoloured vertex with no token left, then Lister wins the game. Otherwise, at the end of some round, all the vertices are coloured and Painter wins the game. We say $G$ is {\em $d$-defective $k$-paintable} if Painter has a winning strategy in this game. We say $G$ is {\em $k$-paintable} if $G$ is $0$-defective $k$-paintable. 
The {\em paint number} $\chi_P$ of $G$ is defined as the mimimum $k$ such that $G$ is $k$-paintable.

It follows from the definition that if $G$ is $d$-defective $k$-paintable then $G$ is $d$-defective $k$-choosable. The converse is not necessarily true.
It was proved in \cite{HZ2016} that every outerplanar graph is $2$-defective $2$-paintable and for every surface $\Sigma$, there is a constant $w$ such that every graph embedded in $\Sigma$ with edge-width at least $w$ is $2$-defective $4$-paintable. In particular, every planar graph is $2$-defective $4$-paintable. 
It was shown in \cite{GHKZ} that every planar graph is $3$-defective $3$-paintable, but there are planar graphs that are not $2$-defective $3$-paintable. The problem whether every planar graph is $1$-defective $4$-paintable remained open for a while, and recently the problem is settled. As a consequence of the main result  in \cite{GZ},  every planar graph is indeed $1$-defective $4$-paintable.

The main result in \cite{GZ} is    about the Alon-Tarsi number of subgraphs of a planar graph.   
 Assume $G$ is a graph.
We associate to each vertex $v$ of $G$ a variable $x_v$. The graph polynomial $P_G(x)$ of $G$ is defined as
\[
P_G(\mathbf{x}) = \prod_{uv \in E(G), \\ u < v} (x_v - x_u)
\]
where $\mathbf{x} = \{x_v : v \in V (G)\}$ denotes the sequence of variables ordered according  to some fixed linear ordering `$<$' of the vertices of G. It is easy to see that a mapping $\phi : V \rightarrow \mathbb{R}$ is a proper colouring of $G$ if and only if $P_G(\phi) \neq 0$, where $P_G(\phi)$ means to evaluate the polynomial at $x_v = \phi(v)$ for $v \in V (G)$. Thus to find a proper colouring of $G$ is equivalent to find an assignment of $\mathbf{x}$ so that the polynomial evaluated at this assignment is non-zero. 

Assume now that $f(\mathbf{x})$ is any real polynomial with variable set $X$. 
Let $\eta$ be a mapping which assigns to each variable $x$ a non-negative integer $\eta(x)$. We denote by $\mathbf{x}^{\eta}$ the monomial $\prod_{x \in X} x^{\eta(x)}$
determined by mapping $\eta$, which we call then the exponent of that monomial. Let $c_{f,\eta}$ denote the coefficient of $\mathbf{x}^{\eta}$ in the expansion of $f(\mathbf{x})$ into the sum of monomials. The Combinatorial Nullstellensatz of \cite{Alon} asserts that if
$\sum_{x \in X} \eta(x) = deg f$ and $c_{f, \eta} \neq 0$, then for arbitrary sets $A_x$ assigned to variables $x \in X$ with $|A_x| \geq \eta(x) +1$, 
there exists a mapping $\phi : X \rightarrow \mathbb{R}$ such that $\phi(x) \in A_x$ for each $x \in X$ and $f(\phi) \neq 0$. 

In particular, Combinatorial Nullstellensatz implies that if $c_{P_G,\eta} \neq 0$ and $\eta(x_v) < k$ for all $v \in V$, then $G$ is $k$-choosable. This  is now a main  tool in the study of list colouring of graphs.  This result was  strengthened by Schauz \cite{Uwe}, who showed that under the same assumptions, the graph $G$ is also $k$-paintable. Jensen and Toft \cite{Toft} defined the Alon-Tarsi number (\emph{AT number} for short) $AT(G)$ of a graph G as
\[
AT(G) = \mbox{min} \{k : c_{P_G, \eta} \neq 0 \mbox{ for some 
} \eta \mbox{ with } \eta(x_v) < k \mbox{ for all } v \in V(G) \}.
\]
As discussed above,   for every graph $G$, 
\[
\chi_{\ell} (G) \leq \chi_P(G) \leq AT(G).
\]
As observed in \cite{Hefetz}, apart from being an upper bound for the choice number and the paint number, the Alon-Tarsi number of a graph has certain distinct features and  is a graph invariant of independent interests.
 
It is known \cite{GK2016} that  the gaps between $AT(G)$ and $\chi_P(G)$, and between $\chi_P(G)$ and $\chi_{\ell}(G)$,  can be arbitrarily large.  Nevertheless, upper bounds for the choice numbers of many  classes of graphs are also upper bounds for their Alon-Tarsi number.  For example, Thomassen \cite{Thomassen} proved that every planar graph is 5-choosable. As a strengthening of this result, it was shown in \cite{Zhu-5} that every planar graph $G$ satisfies $AT(G) \leq 5$. Recently, the following result was   proved in \cite{GZ}.

\begin{theorem}
	\label{thm-gz}
	Every planar graph $G$ has a matching $M$ such that $G-M$ has Alon-Tarsi number at most $4$.
\end{theorem}

This theorem   implies that every planar graph $G$ is $1$-defective $4$-paintable, however, it says something more. To prove that $G$ is $1$-defective $4$-paintable,
we need to show that Painter has a winning strategy in the $1$-defective $4$-painting game. This means that no matter what are Lister's moves, Painter can construct a colouring of $G$,  so that the edges that are not properly coloured form a matching $M$. This matching $M$ depends on Lister's move. However, Theorem \ref{thm-gz} asserts that there is such a matching $M$ that does not depend on Lister's moves. Similarly, to prove that every planar graph   $G$  is $1$-defective $4$-choosable, it amounts to show that for any $4$-list assignment $L$ of $G$, there is a matching $M$ such that $G-M$ is $L$-colourable. In the proof of this result in \cite{CKierstead}, the choice of the matching $M$ depends on $L$. However, Theorem \ref{thm-gz} implies that there is a matching $M$ that works for all $4$-list assignments $L$. 

The result  that every planar graph is $2$-defective $3$-choosable means that for every $3$-list assignment $L$ of a planar graph $G$, there is a subgraph $H$ of $G$ with maximum degree $\Delta(H) \le 2$ such that $G-E(H)$ is $L$-colourable; the result  that every planar graph is $3$-defective $3$-paintable means that for any   Lister's moves in the painting game of a planar graph $G$ with each vertex having $3$ tokens, Painter can colour $G$ so that the edges that are not properly coloured form a subgraph $H$ of  maximum degree $\Delta(H) \le 3$. 
The subgraphs $H$ described above depends on the list assignment $L$ or on the   Lister's moves in the game.  
A natural question is whether there is such a subgraph $H$ that works for all list assignments $L$ or for all   Lister's moves.  Even more ambitiously, one can ask whether every planar graph $G$ has a subgraph $H$ of maximum degree at most $3$ such that $G-E(H)$ has Alon-Tarsi number at most $3$. 

In this paper, we   construct a planar graph $G_1$ such that for any subgraph $H$ of $G_1$ with maximum degree at most $3$, $G_1-E(H)$ is not $3$-choosable. This provides negative answers to all the questions above. 
We also construct a planar graph $G_2$ such that for any star-forest $F$ of $G_2$, $G-E(F)$ contains a copy of $K_4$ and hence is not $3$-colourable. On the other hand, we prove that every planar graph $G$ has a forest $F$ such that $G-E(F)$ has Alon-Tarsi number at most $3$. 
It remains an open problem whether there is a constant $d$, every planar graph has a subgraph (or a forest) $H$ of maximum degree at most $d$ such that $G-E(H)$ is $3$-choosable, or $3$-paintable or has Alon-Tarsi number at most $3$.
If the answer is yes, then what is the smallest such   constant $d$?


\section{Examples of planar graphs}
 Let $J_1$ and $J_2$ be the two graphs depicted in Figure~\ref{fig:J1J2}.
For $i=1,2$, the edge $ab$ in $J_i$ is called the {\em handle} of $J_i$.

	\begin{figure} [t!]
	\centering
	\includegraphics[scale=1.0]{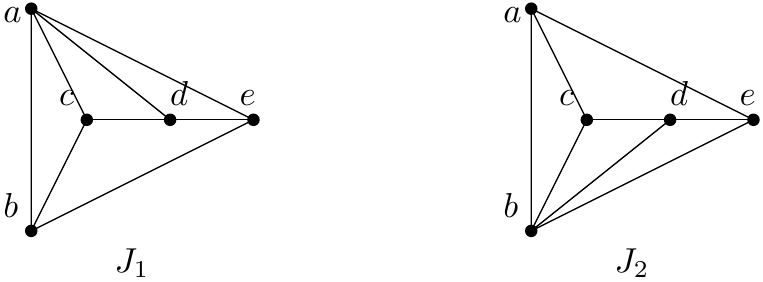}
	\caption{The graphs $J_1$ and $J_2$}
	\label{fig:J1J2}
\end{figure}

Let $\mathcal{J}$ be the set of graphs obtained from the disjoint union of $6$ copies of $J_1$ or $J_2$ by identifying the edges 
corresponding to $ab$ from each copy. 
For each $G\in \mathcal{J}$,
let $c_i,d_i,e_i$ be the vertices corresponding to $c$, $d$, $e$, respectively, for $i\in [6]$, and the edge $ab$ in $G$ is called the {\em handle} of $G$. (See Figure~\ref{fig:J}.)

 	\begin{figure} [t!]
	\centering
	\includegraphics[scale=1.0]{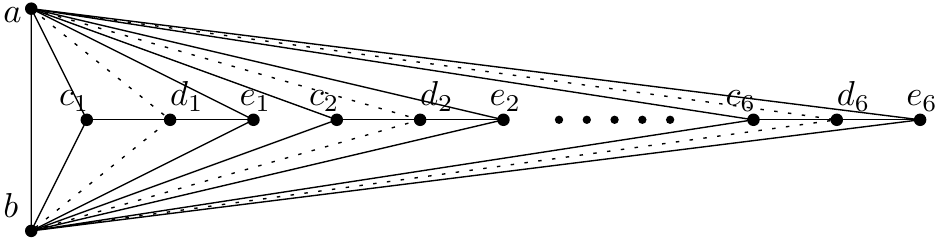}
	\caption{This graph illustrates a graph in $\mathcal{J}$. For each  $i\in [6]$, the subgraph induced on $\{a,b,c_i,d_i,e_i\}$ is isomorphic to $J_1$ or $J_2$ with isomorphism mapping $a,b,c_i,d_i,e_i$ to $a,b,c,d,e$, respectively. That is, there exists exactly one edge between $\{a,b\}$ and $d_i$.}
	\label{fig:J}
	\end{figure}

\begin{lemma}\label{lem:J}
	Every graph $G\in \mathcal{J}$ is not $3$-choosable.
\end{lemma}
\begin{proof}
	We will define a $3$-list assignment  of $G$ using colours $\al$, $\be$, $\ga$ and $\om$ as follows.
	Let $(x_i,y_i, z_i)_{i=1,\ldots,6}$ be the six permutations of the colour set     $\{\al,\be,\ga\}$.
\begin{itemize}
	\item $L(a)=L(b)=\{\al,\be,\ga\}$.
	\item For each $i \in [6]$,   
	$L(c_{i})=\{\al,\be,\ga\}$ and $L(e_{i})=\{x_i,y_i,\om\}$.
	\item 
	For each $i \in [6]$, $L(d_i)=\{x_i,z_i,\om\}$ if $d_i$ is adjacent to $a$ and $L(d_i)=\{y_i,z_i,\om\}$ if $d_i$ is adjacent to $b$.
\end{itemize}
	
	Suppose there exists an $L$-colouring $\phi$ of $G$.
	We may assume that $\phi(a)=\al$ and $\phi(b)=\be$. 
	Without loss of generality, 
	let $x_1=\al$ and $y_1=\be$. Since $L(c_1)=\{\al,\be,\ga\}$ and $L(e_1)=\{x_1,y_1,\om\}=\{\al,\be,\om\}$, we have 
 $\phi(c_1)=\ga$   and $\phi(e_1)=\om$. 
	If $d_1$ is adjacent to $a$, then $L(d_1)=\{\al,\ga,\om\}$ but $\phi(a)=\al$, $\phi(c_1)=\ga$ and $\phi(e_1)=d$, so there is no available colour for $d_1$.
	Similarly, if $d_1$ is adjacent to $b$, then $L(d_1)=\{\be,\ga,\om\}$ but $\phi(b)=\be$, $\phi(c_1)=\ga$ and $\phi(e_1)=\om$, so there is no available colour for $d_1$.
	By the construction of $G$, $d_1$ is adjacent to either $a$ or $b$, therefore, there is no possible colour for $d_1$.
	This leads to a contradiction. Therefore $G$ is not $3$-choosable.\qed
\end{proof}
\bigskip

Let $J_3$ be the graph depicted in Figure~\ref{fig:J3}.
 	\begin{figure} [t!]
	\centering
	\includegraphics[scale=1.0]{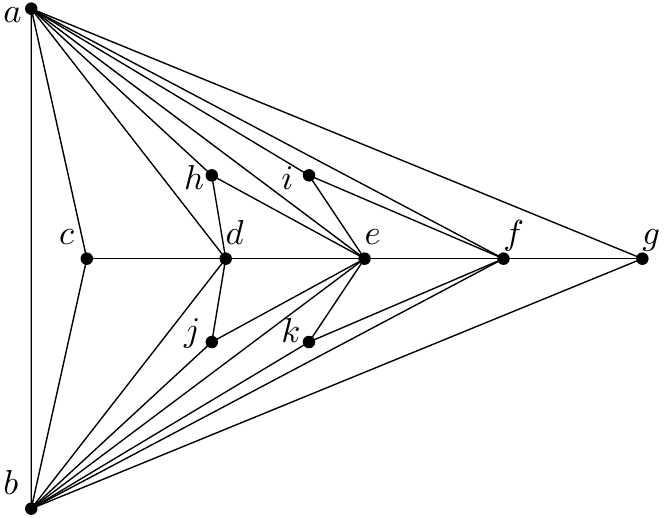}
	\caption{The graph $J_3$.}
	\label{fig:J3}
	\end{figure}

\begin{lemma}\label{lem:J3}
Assume $H$ is a subgraph of $J_3$ with maximum degree at most three.
If $H$ does not contain any edge incident with $a$ or $b$, then $J_3-E(H)$ contains $K_4$, or  a subgraph isomorphic to $J_1$ or $J_2$  with handle $ab$.
\end{lemma}

\begin{proof}
	Assume $H$ is a subgraph of $J_3$ which does not contain any edge incident with $a$ or $b$. If any of the edges $cd,de,ef,fg$ is not contained in $H$, then $J_3-E(H)$ contains $K_4$ and we are done. Thus we assume
		$H$ contains $\{cd,de,ef,fg\}$. 
	If $H$   contains neither $hd$ nor $he$, 
	then the edge set $\{ab,ad,ah,ae,hd,he,bd,be\}$ induces a copy of $J_1$ with handle $ab$ in $J_3-E(H)$ and we are done. 
	So we assume that  $H$ contains   $hd$ or $he$. 
	Similarly, $H$ contains  $ie$ or $if$,  $jd$ or $je$, and $ke$ or $kf$.
	Therefore $|E(H)\cap \{hd,he, ie,if, jd,je, ke,kf\}| \ge 4$.  
	Since every edge in $\{hd,he, ie,if, jd,je, ke,kf\}$ is incident with $d$, $e$ or $f$, 
	it follows from the pigeonhole principle that one of $d$, $e$ and $f$ is an end of at least two edges in $E(H) \cap \{hd,he, ie,if, jd,je, ke,kf\}$.
	It implies that one of $d$, $e$ and $f$ has degree at least four in $H$ because the edges $cd$, $de$, $ef$ and $fg$ are already contained in $H$, 
	which yields a contradiction. 
	This completes the proof.
	\qed
\end{proof}
\bigskip

Let $S$ be the graph obtained from nine copies of $J_3$ by identifying the edges corresponding to $ab$ from each copy. It is obvious that $S$ is a planar graph.
The edge $ab$ in $S$ is called the {\em handle} of $S$.
We obtain the following as a corollary of Lemma~\ref{lem:J3}.
\begin{corollary}\label{cor:S}
Assume $H$ is a subgraph of $S$ with maximum degree at most three. 
If $H$ does not contain any edge incident with $a$ in $S$, 
then $S-E(H)$ contains $K_4$ or a member of $\mathcal{J}$ as a subgraph.
\end{corollary}
\begin{proof}
Let $S_1,S_2,\ldots,S_9$ be the distinct subgraphs of $S$ isomorphic to $J_3$ with handle $ab$.
Since $H$ has maximum degree at most three, and $H$ does not contain any edge incident with $a$ in $S$, 
there are $1\le i_1 < \ldots <i_6 \le 9$ such that every edge incident with $b$ in $S_{i_j}$ is not contained in $H$ for $j\in [6]$.
Without loss of generality, let $i_j=j$ for $j\in [6]$.
Then, by Lemma~\ref{lem:J3}, 
for each $j\in [6]$, $S_j - E(H)$ contains  $K_4$ or a subgraph isomorphic to $J_1$ or $J_2$ with handle $ab$.
We are done if $S_j-E(H)$ contains $K_4$, 
so, we may assume that 
$S_j-E(H)$ has a subgraph $S_j'$ isomorphic to $J_1$ or $J_2$ with handle $ab$.
Then, combining $S_j'$ for $j\in [6]$, we obtain a subgraph of $S-E(H)$  isomorphic to a member of $\mathcal{J}$.
This completes the proof. \qed
\end{proof}
\bigskip

Now, we   construct a graph $G_1$ such that for every subgraph $H$ of $G_1$ with maximum degree at most $3$, $G_1-E(H)$ is not $3$-choosable as follows:
 Start with 
a star with four leaves $v_1,v_2,v_3,v_4$ and center $c$,
and for each $i\in [4]$, we add a copy of $S$ with handle $cv_i$ to the star.
That is, $G_1$ consists of four edge-disjoint copies $S_1,S_2,S_3,S_4$ of $S$ where the handle of $S_i$ is $cv_i$ for $i\in [4]$. 

\begin{theorem} \label{thm-G1}
	For every subgraph $H$ in $G_1$ with $\Delta(H) \leq 3$, $G_1-E(H)$ is not $3$-choosable.
\end{theorem}
\begin{proof}
	Let $H$ be a subgraph in $G_1$ with $\Delta(H) \leq 3$.
	We claim that $G_1-E(H)$ contains $K_4$ or a member of $\mathcal{J}$. 
	Then, by the fact that $K_4$ is not $3$-colourable (so not $3$-choosable) and Lemma~\ref{lem:J}, Theorem~\ref{thm-G1} follows.

	By adding isolated vertices to $H$, we consider $H$ as a spanning subgraph of $G$ with maximum degree at most three.  
	Since $c$ has degree at most three in $H$, there exists $i \in [4]$ such that $H$ does not contain any edge incident with $c$ in $S_i$.
	Then, by Corollary~\ref{cor:S}, $S_i - E(H)$ contains $K_4$ or a member of $\mathcal{J}$. 
	This completes the proof.
	\qed
\end{proof}
\bigskip

	Next, we   show that there is a planar graph $G_2$ such that for any star-forest $F$ of $G_2$, $G_2-E(F)$ contains $K_4$ and hence is not $3$-colourable.

	Let $A$ be the graph depicted in Figure~\ref{fig:A}. In $A$, the edge $xy$ is called the {\em handle} of $A$.
Assume $F$ is  a star forest. By a \emph{center} of $F$, we mean the center of some component of $F$. If $K_2=uv$ is a component of   $F$, we arbitrarily choose one of $u,v$ as the center.

	\begin{figure} [t!]
		\centering
		\includegraphics[scale=1.0]{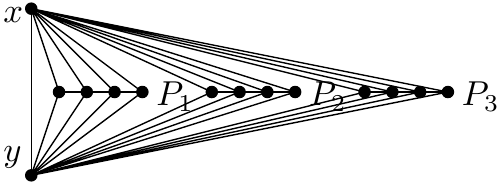}
		\caption{The graph A  consists of three disjoint paths $P_1$, $P_2$, $P_3$ on $4$ vertices and adjacent vertices $x$ and $y$ such that $\{x,y\}$ is complete to $\bigcup_{i=1}^3 V(P_i)$.}
		\label{fig:A}
	\end{figure}

	\begin{lemma}\label{lem:saw}
		Assume $F$ is a star forest in $A$.
		If neither $x$ nor $y$ is a center of $F$, 
		then $A-E(F)$ contains $K_4$.
	\end{lemma}
	\begin{proof}
		Since every center of $F$ is contained in $\bigcup_{i=1}^3 V(P_i)$,
		there is some $i\in [3]$ such that none of the edges between $\{x,y\}$ and $V(P_i)$ is contained in $F$.
		Since $F$ does not contain a path on $4$ vertices, there is an edge $uv \in E(P_i)-E(F)$.
		Then, all edges with both ends in $\{x,y,u,v\}$ remain in $G-E(F)$, and they induce $K_4$.
		This proves Lemma~\ref{lem:saw}. \qed
	\end{proof}
	\bigskip

\begin{theorem}\label{thm:example}
	There exists a planar graphs $G_2 $  such that 
	for every star forest $F$ in $G_2 $, $G_2 -E(F)$ contains $K_4$ 
\end{theorem}
\begin{proof}
	Let $D$ be the graph depicted in Figrue~\ref{fig:D}.
	\begin{figure} [ht!]
		\centering
		\includegraphics[scale=0.8]{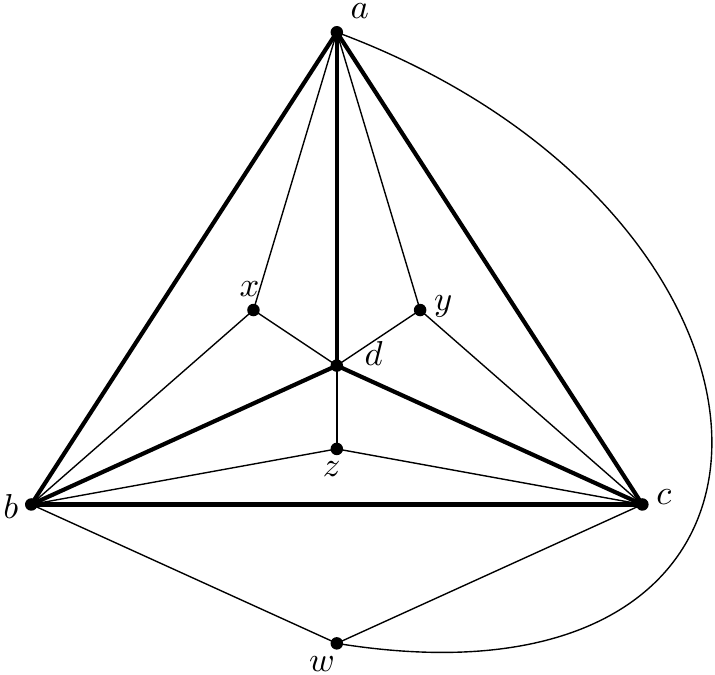}
		\caption{The graph $D$.}
		\label{fig:D}
	\end{figure}
	We construt  $G_2$ from $D$ by attaching,  for each edge $e$ of $D$, a copy of $A$ with handle $e$ to  $D$.
	
	By the construction of $G_2$, 
	we know that $G_2$ contains $18$ edge-disjoint copies of $A$ where the handle of each copy is an edge of $D$. 
	For each $e\in E(D)$, let $A_e$ be the copy of $A$ in $G_2$ with handle $e$.

	Suppose, for the sake of contradiction, that there exists a star forest $F$ in $G_2$ such that $G_2-E(F)$ does not contain $K_4$.
	
	For every edge $e$ of $D$, at least one end of $e$ is a center of $F$, since otherwise, 
	$A_e-E(F) (\subseteq G_2-E(F))$ contains $K_4$  by Lemma~\ref{lem:saw}.
	Since $\{a,b,c,d\}$ induces $K_4$, at least three of them are centers of $F$. 
	Without loss of generality, we assume that $a$, $b$ and $c$ are centers of $F$.
	Then, the edges $ab$, $bc$ and $ca$ are not contained in $F$ since there is no edge in a star forest joining two centers.
	If none of $\{ad,bd,cd\}$ is contained in $F$, then $\{a,b,c,d\}$ induces $K_4$ in $G_2-E(F)$.
	Hence, $\{ad,bd,cd\} \cap E(F)\neq \emptyset$. This implies that $d$ is not a center
since  $a,b,c$ are centers of $F$,
	so, exactly one of $\{ad,bd,cd\}$ is contained in $F$. 
	Without loss of generality, we may assume that $ad \in E(F)$. 
	
	Now, since $d$ is not a center of $F$, $z$ must be a center of $F$.  So, $bz$ and $cz$ are not contained in $F$.  And since $ad$ is in $F$ and $d$ is not a center of $F$, $dz$ is not contained in $F$.  
	Thus every edge with both ends in $\{b,c,d,z\}$ is not contained in $F$.
	Therefore, $\{b,c,d,z\}$ induces $K_4$ in $G_2 -E(F)$.
	This completes the proof of Theorem \ref{thm:example}.  \qed
\end{proof}
 \medskip

\section{Alon-Tarsi number of a planar graph minus a forest}

We say a digraph $D$ is \emph{Eulerian} if $d^+_D(v) = d^{-}_D(v)$ for every vertex $v$. Assume $G$ is a graph and $D$ is an orientation of $G$. Let $EE(D)$(respectively, $OE(D)$) be the set of spanning Eulerian sub-digraphs of $D$ with an even (respectively, an odd) number of edges. An orientation $D$ of a graph $G$ is called {\em an Alon-Tarsi orientation} if $|EE(D)| \neq |OE(D)|$.
Alon and Tarsi \cite{AT92} proved that if $D$ is an orientation of $G$, and $\eta(x_v) = d^+_D(v)$, then $|c_{P_G, \eta}|$ is equal to the absolute value of the difference $|EE(D)| -|OE(D)|$. Hence the Alon-Tarsi number of $G$ can be defined alternatively as the minimum integer $k$ such that $G$ has an Alon-Tarsi orientation $D$ with     $d_D^+(v) < k$ for every vertex $v \in V(G)$.

\begin{theorem}\label{thm:main}
	For every planar graph $G$, there exists a forest $F$ in $G$ such that $G-E(F)$ has AT number at most three.
\end{theorem}

\begin{definition} \label{nice-orientation} \rm
Assume $G$ is a plane graph, $e=xy$ is a boundary edge of $G$, and $F$ is a forest in $G$ containing $e$.
An orientation $D$ of $G' = G -  E(F)$ is {\em nice} for $(G, e, F)$ if the following hold:
\begin{itemize}
\item[(1)] $|EE(D)| \neq |OE(D)|$,

\item[(2)] $d_D^{+}(x) = d_D^{+}(y) = 0$, \ $d_D^{+}(v) \leq 1$ for every boundary vertex $v$ of $G$,  and $d_D^{+}(u) \leq 2$ for every interior vertex $u$ of $G$. 
\end{itemize}
\end{definition}


Note that it is obvious that Theorem \ref{thm:main} follows from Lemma \ref{main-lemma} below.

\begin{lemma} \label{main-lemma}
Assume $G$ is a plane graph of which boundary is a simple cycle.
Suppose every interior face of $G$ is a triangle.
Then, for any boundary edge $e=xy$ of $G$, there exists a forest $F$ in $G$ containing $e$ 
such that  $G' = G -  E(F)$ has a nice orientation $D$ for $(G,e,F)$.
\end{lemma}
\begin{proof}
We prove the lemma by induction on $|V(G)|$.  
It is trivial when $|V(G)|=3$.

Assume $|V(G)|>3$.
Let $C$ be the boundary cycle of $G$.
We consider the following two cases. 

\medskip

\noindent {\bf Case 1}:  $C$ has a chord $uv$.

There are two internally disjoint paths $P_1$ and $P_2$ from $u$ to $v$ in $C$.
For $i=1,2$, let $C_i$ be the cycle consisting of $P_i$ and $uv$,
and $G_i$ be the plane subgraph of $G$ bounded by $C_i$. 
Without loss of generality, we assume $xy \in E(G_1)$.
Clearly, $G_1$ and $G_2$ are plane graphs with interior faces being triangles, their boundaries are simple cycles, and furthermore, $e(G_1),e(G_2) < e(G)$.
So, we can apply the induction hypothesis to $(G_1,xy)$ and $(G_2,uv)$, 
and we obtain forests $F_1$ in $G_1$ and $F_2$ in $G_2$ such that
\begin{itemize}
\item $xy \in E(F_1)$, and $G_1-E(F_1)$ has a nice orientation $D_1$ for $(G_1,xy,F_1)$,  and
\item $uv \in E(F_2)$, and $G_2 -E(F_2)$ has a nice orientation $D_2$ for $(G_2,uv,F_2)$.
\end{itemize}

Let $G_2'=G_2-uv$, and $F_2'=F_2-uv$.
Then, clealry, $G_2-E(F_2)=G_2'-E(F_2')$, and we can consider  
 $D_2$ as an orientation of $G_2'-E(F_2')$.
Clearly, $G_1$ and $G_2'$ are edge-disjoint, and $F_1$ and $F_2'$ are edge-disjoint.
So,  
 $D=D_1\cup D_2$ is well-defined and forms an orientation of $G (=G_1\cup G_2')$, and
since $u$ and $v$ are contained in distinct components of $F_2'$,
it follows that $F=F_1 \cup F_2'$ is a forest  containing $xy$.
We claim that $D$ is a nice orientation of $G-E(F)$ for $(G,xy,F)$.

Since $d_{D_2}^+(u) = d_{D_2}^+(v) = 0$, there is no directed path from $u$ to $v$ or from $v$ to $u$ in $D_2$.
Therefore, there is no directed cycle in $D$ containing both an edge in $D_1$ and an edge in $D_2$.
This implies that every spanning Eulerian sub-digraph of $D$ can be decomposed into a spanning Eulerian sub-digraph of $D_1$ and a spanning Eulerian sub-digraph of $D_2$.
Therefore, we have 
\[
\begin{array}{ccc}
|EE(D)| & = & 
\big(|EE(D_1)| \times |EE(D_2)|\big) + \big(|OE(D_1)| \times |OE(D_2)|\big) \\
|OE(D)| & = & 
\big(|EE(D_1)| \times |OE(D_2)| \big) +  \big( |OE(D_1)| \times |EE(D_2)| \big).
\end{array}
\]
So, we have
\[
|EE(D)| - |OE(D)| = (|EE(D_1)| - |OE(D_1)|) \cdot (|EE(D_2)| - |OE(D_2)|). 
\]
Now, since $|EE(D_1)| - |OE(D_1)| \neq 0$ and  $|EE(D_2)| - |OE(D_2)| \neq 0$, we conclude that $|EE(D) - OE(D)| \neq 0$.
Therefore, $D$ satisfies the condition (1) in the definition of a nice orientation for $(G,e,F)$.
Next, since $d_{D_2}^+(u) = d_{D_2}^+(v) = 0$, it is easily checked that $D$ satisfies Condition (2) of Definition \ref{nice-orientation}.

\bigskip
\noindent {\bf Case 2}:   $C$ has no chord. 

Let $z(\neq y),w$ be the vertices of $C$ such that $x, z, w$ are consecutive in $C$. 
(It is possible that $w=y$.)
Since every interior face of $G$ is a triangle, the neighbors of $z$ in $G$ forms a path $P$ from $x$ to $w$.
 Since $|V(G)|>3$ and $G$ has no chord, $P$ has length at least two.
Now we consider the graph $G'=G-z$.

Clearly, $G'$ is a plane graph with interior faces being triangles and its boundary is the cycle obtained from $C$ by removing $z$ and adding $P$. 
Furthermore, $e(G')<e(G)$.
Hence, by the induction hypothesis, there is a foreset $F'$ in $G'$ containing $xy$
where
$G'-E(F')$ has a nice orientation $D'$ for  $(G',xy,F')$.
Let $F=F' \cup \{zw\}$. 
Clearly, $F$ is a forest in $G$ containing $xy$.
We extend $D'$ to an orientation of $G-E(F)$ by adding 
$\{\overrightarrow{zx}\}$ and $\{\overrightarrow{uz}\mid u \in V(P)\setminus \{x,w\}\}$.
We claim that $D$ is a nice orientation of $G-E(F)$ for $(G,xy,F)$.

There are no added arcs going out from vertex in $V(C)$ except $\overrightarrow{zx}$, 
so
$d^+_D(x)=d^+_D(y)=0$ and $d^+_D(u) \le 1$ for every vertex $u\in V(C)$ since $z$ has only one outgoint edge in $D$. Since $d_{D'}(u) \leq 1$ for $u \in V(P) \setminus \{x, w\}$, we have $d^+_{D}(u) \le 2$ for every $u \in V(G') \setminus V(C)$.
Therefore, the condition (2) holds.

Next, we will show that  Condition (1) of Definition \ref{nice-orientation}  holds.
Let $H$ be a directed cycle in $D$.  
If $H$ contains any arc from $V(P)$  to $z$, then $H$ must pass $x$ since 
$x$ is the only out-neighbor of $z$.  But, this is not possible since $d_{D}^+(x)  = 0$.  
Therefore, $H$ does not contain any edges which is incident to $z$.
This means that $H$ is a directed cycle in $D'$.  
Thus, every spanning Eulerian sub-digraph of $D$ consists of the isolated vertex $z$ and a spanning Eulerain sub-digraph of $D'$, 
so 
$|EE(D)|=|EE(D')|$ and $|OE(D)|=|OE(D')|$, and by the assumption that $|EE(D')|\neq |OE(D')|$,
we have $|EE(D)|\neq |OE(D)|$.
Therefore, $D$ is a nice orientation of $G-E(F)$ for $(G,xy,F)$.
This completes the proof. \qed
\end{proof}

  \bigskip

\medskip
 
The following question remains open.

\begin{question}
	Is there a constant $d$ such that   every planar graph $G$ has a forest $F$ of maximum degree at most $d$ such that $G-E(F)$ has Alon-Tarsi number at most $3$? If so, what is the smallest $d$?
\end{question} 



\begin{thebibliography}{00}

\bibitem{Alon} N. Alon. Combinatorial nullstellensatz. Combinatorics, Probability, and Computing, 8:7--29, 1999.

\bibitem{AT92} N. Alon and M. Tarsi. Colorings and orientations of graphs. Combinatorica, 12:125--134, 1992.


 
\bibitem{CCW} L. J. Cowen, R. H. Cowen, and D. R. Woodall. Defective colorings of graphs in surfaces: Partitions into subgraphs of bounded valency. Journal of Graph Theory, 10(2):187--195, 1986.

\bibitem{CKierstead} W. Cushing and H. A. Kierstead. Planar graphs are 1-relaxed, 4-choosable. European Journal of Combinatorics, 31(5):1385--1397, 2010.

\bibitem{EH} N. Eaton and T. Hull. Defective list colorings of planar graphs. Bulletin of the Institute of Combinatorics and its Applications, 25:79--87, 1999.

\bibitem{GHKZ} G. Gutowski, M. Han, T. Krawczyk, and X. Zhu. Defective 3-paintability of planar graphs. Electron. J. Combin., 25(2):Paper 2.34, 20, 2018.

\bibitem{GK2016} G. Gutowski and J. Kozik,  
Chip games and paintability. (English summary) 
Electron. J. Combin. 23 (2016), no. 3, Paper 3.3, 12 pp. 


\bibitem{GZ} 
J. Grytczuk and X. Zhu.  The Alon-Tarsi number of a planar graph minus a matching, arXiv:1811.12012.

\bibitem{HZ2016} M. Han and X. Zhu,  Locally planar graphs are 2-defective 4-paintable, European J. Combin. 54 (2016), 35–50.

\bibitem{Hefetz} D.
Hefetz, 
On two generalizations of the Alon-Tarsi polynomial method. 
J. Combin. Theory Ser. B 101 (2011), no. 6, 403–414. 


\bibitem{Toft} T. Jensen and B. Toft. Graph Coloring Problems. Wiley, New York, 1995.


\bibitem{Uwe} U. Schauz. Mr. Paint and Mrs. Correct. Electronic Journal of Combinatorics, 16(1):R77:1--18, 2009.

\bibitem{Thomassen} C. Thomassen. Every planar graph is 5-choosable. Journal of Combinatorial Theory, Ser. B, 62(1):180--181, 1994.


\bibitem{Skrekovski1999} R. Skrekovski. List improper colourings of planar graphs. Combinatorics, Probability and Computing, 8(3):293--299, 1999.

\bibitem{Zhu-5} X. Zhu. Alon-Tarsi number of planar graphs. Journal of Combinatorial Theory Ser. B, https://doi.org/10.1016/j.jctb.2018.06.00.




\end{thebibliography}
\end{document}